\documentclass[12pt,a4paper,reqno]{amsart}
\usepackage{amsmath,amssymb,amsthm,amscd,mathrsfs} 
\usepackage[T1]{fontenc}
\usepackage[latin1]{inputenc}
\usepackage[english]{babel}
\usepackage[mathscr]{eucal}
\usepackage{graphicx}
\usepackage{float}
\usepackage{tikz}
\usetikzlibrary{positioning}
\usetikzlibrary{shapes.geometric}
\usetikzlibrary{fit}
\usetikzlibrary{patterns}
\usepackage{caption}
\usepackage{tikz-cd}
\usepackage{cases}
\usepackage{hyperref}
\usepackage{chngcntr}

\usetikzlibrary{decorations.text}

\input{comment.sty}
\includecomment{FZ}
\includecomment{MM}
\counterwithin{equation}{section}

\renewcommand{\a}{\alpha}

\renewcommand{\th}{\theta}

\renewcommand{\k}{\kappa}

\renewcommand{\l}{\lambda}

\newcommand{\n}{\nu}

\newcommand{\ph}{\phi}
\def\ph{\phi}
\def\md#1{\ \mbox{\rm(mod }{#1})}

\def\npp#1{N_{\ph}^+(#1)}
\def\ol{\overline}

\def\ph{\phi}

\newcommand{\Q}{{\mathbb Q}}
\newcommand{\Z}{{\mathbb Z}}

\newcommand{\F}{{\mathbb F}}

\def\md#1{\ \mbox{\rm(mod }{#1})}

\def\npp#1{N_{\ph}^+(#1)}

\newcommand{\aF}{\mathfrak a}
\newcommand{\bF}{\mathfrak b}

\newcommand{\pF}{\mathfrak p}

\newcommand{\qF}{\mathfrak q}

\newtheorem{theorem}{Theorem}[section]
\newtheorem{proposition}[theorem]{Proposition}
\newtheorem{lemma}[theorem]{Lemma}
\newtheorem{corollary}[theorem]{Corollary}

\theoremstyle{definition}

\theoremstyle{remark}
\newtheorem{remark}[theorem]{Remark}

\newtheorem{example}[theorem]{Example}
\begin{document}
\title[]{ On non-monogenic number fields defined by trinomials of  type $x^n +ax^m+b$}
\textcolor[rgb]{1.00,0.00,0.00}{}
\author{Hamid Ben Yakkou }\textcolor[rgb]{1.00,0.00,0.00}{}
\address{Faculty of Sciences Dhar El Mahraz, P.O. Box  1874 Atlas-Fes , Sidi mohamed ben Abdellah University,  Morocco}\email{beyakouhamid@gmail.com}
\maketitle
\vspace{0.3cm}
\begin{abstract}  Let $K=\Q(\theta)$ be a number field generated  by a complex  root  $\th$ of a monic irreducible trinomial  $F(x) = x^n+ax^{m}+b \in \Z[x]$.    In this paper, we deal with the problem of the non-monogenity of $K$.  More precisely, we provide some explicit conditions on $a$,  $b$, $n$, and $m$ for which $K$ is not monogenic. As  application, we show that there are infinite families of non-monogenic number fields defined by trinomials of degree $n=2^r\cdot3^k$ with $r$ and $k$ are positive integers.  We also give  two infinite families of non-monogenic number fields defined by trinomials of degree $6$. Finally, we illustrate our results by giving  some examples.  
	       
\end{abstract}
\maketitle

\section{Introduction}

Let $K=\Q(\th)$ be a number field  generated by a complex  root   $\th$ of a monic irreducible  polynomial $F(x)$ of degree $n$ over  $\Q$ and   $\Z_K$   its  ring of  integers. The field $K$ is called monogenic if there exists  $\eta \in \Z_K$ such that $\Z_K= \Z[\eta]$, that is $(1, \eta, \ldots, \eta^{n-1})$ is an integral basis (called a power integral basis) in $K$. The polynomial $F(x)$ is called monogenic if $\Z_K = \Z[\th]$, or equivalently if  $(1, \th, \ldots, \th^{n-1})$ is an integral basis  in $K$.  It is important to note that the monogenity of  the polynomial $F(x)$ implies the monogenity of the field $K$. But, the non-monogenity of $F(x)$ does not imply the non-monogenity of $K$. For an example,   refer to \cite[Theorem 2.1]{BFT1} when we gave a family of monic irreducible polynomials of type $x^{p^r}+ax+b$ which are not monogenic, but their roots generate monogenic number fileds.   There are extensive  computational results   concerning  the problem of   monogenity of number fields and constructing power integral bases.   S. Ahmad, T. Nakahara and S. M. Husnine  proved in \cite{ANHN6}  that if $m \equiv 1 \md 4$ and $m \not \equiv  \pm 1 \md 9 $ is a square-free rational integer, then the sextic pure field $K = \Q(\sqrt[6]{m})$  cannot be monogenic.  In \cite{GP}, Ga\'{a}l,  Peth\H{o}, and Pohst studied the monogenity of certain quartic number fields. In\cite{Gacta1}, Ga\'{a}l and Gy\H{o}ry described an algorithm to solve  index form equations in quintic fields  and they computed all generators of power integral bases in some totally real quintic fields with Galois group $S_5$. In  \cite{Gacta2}, Bilu, Ga\'{a}l, and Gy\H{o}ry   studied  the monogenity of  some  totally real sextic fields with Galois group $S_6$.   In \cite{GR17}, Ga\'{a}l and Remete  obtained new deep results on   monogenity of pure number fields $K= \Q(\sqrt[n]{m})$ for $3 \le n \le 9$ and any  square-free rational integer  $m \neq  \pm 1$. They also  showed in \cite{GRN} that if $m \equiv  2\, \mbox{or} \,3 \md 4$ is a square-free rational integer, then the octic field $K= \Q(i, \sqrt[4]{m} )$ is not monogenic.  In \cite{PP}, Peth\H{o} and Pohst studied the indices in some multiquadratic number fields. In \cite{Smtacta},  Smith studied the monogenity of radical extensions and   gave  sufficient conditions for a  Kummer extension $\Q(\xi_n, \sqrt[n]{\a})$  to be not monogenic.    In \cite{BF, BFC}, Ben Yakkou et al. considered the problem of monogenity  in certain pure number fields  with large degrees.   Let 
 \begin{eqnarray}\label{i(K)}
 	i(K) = \gcd \ \{ ( \Z_K : \Z [\eta]) \, |\,\eta \in \Z_K \,and \, K= \Q(\eta)  \}
 \end{eqnarray} be the index  of the field  $K$. In \cite{DS}, Davis and Spearman calculated the index of the quartic field defined by $x^4+ax+b$. El Fadil gave in \cite{Fcom} necessary and sufficient conditions on $a$ and $b$ so that a rational prime integer $p$ is a common index divisor of number fields defined by $x^5+ax^2+b$.   Jakhar and Kumar in \cite{JK} gave infinite families of  non-monogenic number fields defined by $x^6+ax+b$. In \cite{Ga21}, Ga\'al studied the multi-monogenity of sextic number fields defined by trinomials of type $x^6+ax^3+b$.   For the same number fields studied by  Ga\'{a}l, El Fadil gave a complete characterization of the prime  divisors of the index of these number fields.
 Also, in \cite{B}, Ben Yakkou studied the monogenity of certain number fields defined by $x^8+ax+b$.
The purpose of this paper is to  study the problem of the non-monogenity  of certain number  fields  $K= \Q(\th)$ generated by a complex  root  $\th $ of a monic irreducible   trinomial of type  $F(x)= x^n +ax^m+b$  when $ \Z_K \neq \Z[\th]$.  
Recall that the problem of integral clossedness of $\Z_K$ has  been  studied  in \cite{JKS} by Jakhar et al. Their results are refined in \cite{Smt} by Ibara et al. Also, in  \cite{JW, JP}, Jones et al. identified new infinite families of monogenic trinomilas. More precisely, they gave necessary and sufficient conditions involving $a$, $b$, $n$, and $m$ for $\Z_K$ to be equal $\Z[\th]$.  Also, the non-monogenity of $K$ in the special  case $m=1$ has been previously studied in \cite{BFT1} by Ben Yakkou and El Fadil.  It is important to note  that the fundamental method which allows to test whether a number field is monogenic or not is to solve the index form equation which is very complicated for higher number field degrees (see for example \cite{G19, Ga21, Gacta1, Gacta2, GR17}). For this reason, we have   based our method  on Newton polygon techniques applied on  prime ideal factorization which is an efficient tool to investigate the monogenity of fields defined by trinomials.
\section{Main Results}
 In the remainder of this section, $K=\Q(\th)$ is a number field generated  by a complex  root  $\th$ of a monic irreducible trinomial of  type  $F(x) = x^n + ax^m +b \in \Z[x]$. 
 Let $p$ be  a rational prime integer. Throughout this paper, $\F_p$ denotes the finite field with $p$ elements. For $t \in \Z$, $\n_p(t)$ stands for the $p$-adic valuation of $t$  and let $t_p=\frac{t}{p^{\nu_p(t)}}$. For two positive rational integers $d$ and $s$,  we shall denote by $N_p(d)$ the number of  monic irreducible polynomials of degree $d$ in $\F_p[x]$, $N_p(d,s,t)$ the number of monic irreducible factors of degree $d$ of the polynomial $x^s+\overline{t}$ in $\F_p[x]$, and $N_p(d,s,t)[m,c]$ the number of monic irreducible factors of degree $d$ of  $x^s+\overline{t}$ in $\F_p[x]$ which do not divide  $x^m+\ol{c}$. 
 It is known from \cite{Greenfield} that  the discriminant of the trinomial $F(x) = x^n +ax^m +b $ is 
 \begin{eqnarray}\label{disctrinom}
 	\Delta(F)= (-1)^{\frac{n(n-1)}{2}} b^{m-1} (n^{n_1} b^{n_1-m_1} - (-1)^{m_1} m^{m_1}(m-n)^{n_1 - m_1}a^{n_1})^{d_0},
 \end{eqnarray}
 where $d_0 = \gcd(n , m)$, $n_1 = \frac{n}{d_0} $, and $m_1= \frac{m}{d_0}$.
 It follows by (\ref{indexdiscrininant}) and (\ref{i(K)}) that  if a rational  prime integer  $p$ divides $i(K)$, then  $p^2$ divides $ \Delta(F)$. Also, by Zylinski's condition, if $p$ divides $i(K)$, then $p < n$ (see \cite{Zylinski}). Not that these two conditions on $p$ have been taken into consideration  to find  suitable hypothesis of  our results. Also, for the simplicity of calculation, we use the fact  that for any  rational  prime integer $p$ and any rational integer $b$, $\n_p(b+(-b)^{p^r})= \n_p(b^{p-1}-1)$ for every positive integer $r$ (see \cite{BFC}).

   In what follows, we give some  sufficient conditions  on $a$, $b$, $n$, and $m$ for which such fields are not monogenic.  Note that in this direction, our results   with the  significant results  given in \cite{JKS,  JW, JP} give a deep investigation on the monogenity of such number fields.
\begin{theorem}\label{dn1}
	Let  $p$ be an odd rational prime integer such that $p \mid a$,  $p \nmid b$, and $p \mid n$. Set $n = s\cdot p^r$ with $p \nmid s$. Let $\mu = \n_p(a)$ and  $\nu =\n_p(b^{p-1}-1)$. If for some positive integer $d$, one of the following conditions holds:
\begin{enumerate} 
\item $\mu < \min\{ \nu, r+1\}$ and $N_p(d)<\mu N_p(d,s,b)$, 
\item  $\nu < \min\{ \mu, r+1\}$ and $N_p(d)<\nu N_p(d,s,b)$, 
\item  $\mu = \n \le r  $ and $N_p(d)< \mu N_p(d,s,b)[m, \frac{b+(-b)^{p^r}}{a}]$, 
\item  $ r+1  \le \min\{ \nu, \mu\}$ and $N_p(d)<(r+1) N_p(d,s,b)$, 

\end{enumerate}	
 	then $K$ is not monogenic.
\end{theorem}
\begin{corollary}$($\cite[Theorem 2.2]{Fcom}$)$\\
Let $p$ be an odd rational prime integer and  $F(x)= x^{p^r}+ax^m+b$. If $a \equiv 0 \md {p^{p+1}}$, $b^{p-1} \equiv 1 \md{p^{p+1}}$, and $r \ge p$, then $K$ is not monogenic. In particular, for $p=3$, if $a \equiv 0 \md{81}$, $b \equiv \pm 1 \md{81}$, and $r \ge 3$, then $K$ is not monogenic.
\end{corollary}

\begin{remark}
	Theorem \ref{dn1} also implies \cite[Theorem 2.4]{BF} and \cite[Theorems 2.2 and  2.5]{BFT1}, where  the special cases  $a=0$ and  $m=1$ are respectively previously studied.
	
\end{remark}
\begin{corollary}\label{cordn1}
For $F(x)=x^{2^k \cdot 3^r}+ax^m+b$. If one of the following conditions holds:
\begin{enumerate}
\item $k \ge 1, r \ge 2, a \equiv 9, 18 \md {27} \mbox{and} \, b \equiv -1 \md {27}$, 
\item $k\ge 1$, $r \ge 3$, $a \equiv 27, 54 \md{81}$ and $b \equiv -1 \md {81}$, 
\item $k \ge 1, r \ge 2, a \equiv 0 \md {27} \mbox{and} \, b \equiv 8, 17 \md {27}$, 
\item $k \ge 1, r \ge 3, a \equiv 0 \md {81} \mbox{and} \, b \equiv 26, 53 \md {81}$, 
\item $k \ge 1, r = 1, a \equiv 0 \md 9 \mbox{and} \, b \equiv  -1  \md {9}$, 
\item $k \ge 1, r = 2, a \equiv 0 \md {27} \mbox{and} \, b \equiv  -1  \md {27}$, 
\item $k = 1, r \ge 7, a \equiv 3^7, 2 \cdot 3^7 \md {3^8} \mbox{and} \, b \equiv  1  \md {3^8}$, 
\item $k = 1, r \ge 7, a \equiv 0 \md {3^8} \mbox{and} \, b \equiv  1+3^7, 1+2\cdot3^7  \md {3^8}$, 
\item $k = 1, r =6, a \equiv 0 \md {3^7} \mbox{and} \, b \equiv  1 \md {3^7}$, 
\item $k = 2, r \ge 4, a \equiv 81, 162 \md {243} \mbox{and} \, b \equiv 1  \md {243}$, 
\item $k = 2, r \ge 4, a \equiv 0 \md {243} \mbox{and} \, b \equiv 82, 163  \md {3^5}$,
\item $k = 2, r =3, a \equiv 0 \md {81} \mbox{and} \, b \equiv  1 \md {81}$, 

\end{enumerate}
then $K$ is not monogenic
\end{corollary}
 Remark that if we fix $k=1$ in Corollary \ref{cordn1}(5), then we conclude that if $a$ and $b+1$ are divisible by $9$, then $K$ is not monogenic. The following  theorem gives   two  infinite families  of non-monogenic  sextic number fields defined by trinomials with partial informations about their index.
\newpage
\begin{theorem}\label{d6}
Let $K=\Q(\th)$ be a sextic	 number field generated by a complex root 	 of a monic irreducible trinomial $x^6+ax^m+b$. Then the following hold:
\begin{enumerate}
\item  If $a\equiv 0 \md 9$ and $b \equiv -1 \md 9$, then $i(K)\equiv 3\, \mbox{or} \,6 \md 9$.
\item   If $a\equiv 0 \md 8$ and $b \equiv -1 \md 8$, then $i(K) \equiv 4 \md 8$.
\end{enumerate}

 In particular, if one of these conditions holds, then $K$ is not monogenic.
\end{theorem}

\begin{remark}
Our  Theorem  \ref{d6} generalize Jakhar's and Kumar's result given in  \cite[Theorem 1.1]{JK} when they studied  the special  case $m=1$.

\end{remark}

\begin{theorem}\label{dn2}
	Let $p$ be an odd rational prime integer such that $p\nmid a$, $p\mid b$, and $p\mid n-m$. Set $n-m = u \cdot p^k$ with  $p \nmid u$. Let $\delta = \n_p(b)$ and  $\k = \n_p(a^{p-1}-1)$. If for some positive integer $d$, one of the following conditions holds:
	\begin{enumerate}
	\item  $\delta < \min \{\k, k+1\}$ and $N_p(d) < \delta N_p(d,u,a)$, 
	\item  $\k < \min \{ \delta, k+1\}$ and $N_p(d) < \k N_p(d,u,a)$, 
	 
	 \item    $\k = \delta \le k  $ and $N_p(d)< \k N_p(d,s,b)[m, \frac{b}{a+(-a)^{p^k}}]$, 
	 \item   $k+1 \le \min \{\k, \delta\}$ and $N_p(d) < (k+1) N_p(d,u,a)$, 
	\end{enumerate}
	then $K$ is not monogenic.

\end{theorem}
\begin{remark}
Theorem \ref{dn2} implies \cite[Theorem 2.7]{BFT1} when the case $m=1$ is previously considered.

\end{remark}
\begin{corollary}\label{cordn2}
For $F(x)= x^{(s+1)\cdot 2^r \cdot 3^k}+ax^{s\cdot2^r\cdot 3^k}+b$ with $s$ is a positive rational integer. If one of the following conditions holds:
\begin{enumerate}
\item $r=0$, $k \ge 5$, $a \equiv \pm1 \md{243}$,  and $b\equiv 81, 162 \md{243}$, 
\item $r=0$, $k \ge 5$, $a \equiv 80, 82, 161, 163  \md{243}$,  and $b\equiv 0 \md{243}$, 
\item $r=0$, $k=3$, $a \equiv \pm1  \md{81}$,  and $b\equiv 0 \md{81}$, 
\item $r\ge 1$, $k\ge 2$, $a \equiv -1  \md{27}$,  and $b\equiv 9, 18 \md{27}$, 
\item $r\ge 1$, $k\ge 2$, $a \equiv 8, 17  \md{27}$,  and $b\equiv 0 \md{27}$, 
\item $r\ge 1$, $k=1$, $a \equiv -1  \md{9}$,  and $b\equiv 0 \md{9}$, 
\item $r = 1$, $k\ge 7$, $a \equiv 1  \md{3^8}$,  and $b\equiv 3^7, 2 \cdot 3^7 \md{3^8}$, 
\item $r = 1$, $k\ge 7$, $a \equiv 1+3^7, 1+2\cdot3^7  \md{3^8}$,  and $b\equiv 0 \md{3^8}$, 
\item $r = 1$, $k=6$, $a \equiv 1  \md{3^7}$,  and $b\equiv 0 \md{3^7}$, 
\item $r = 2$, $k\ge 5$, $a \equiv 1  \md{243}$,  and $b\equiv 81, 162 \md{243}$, 
\item $r = 2$, $k\ge 5$, $a \equiv 82, 163  \md{243}$,  and $b\equiv 0 \md{243}$, 
\item $r = 2$, $k=3$, $a \equiv 1  \md{243}$,  and $b\equiv 0 \md{243}$, 
\end{enumerate}
then $K$ is not monogenic.
\end{corollary}

\begin{corollary}
Under the hypothesis of Theorem \ref{dn2}. If $ \gcd\{ \delta, m\} = 1$, and one of the following conditions holds: 
	\begin{enumerate}
	\item  $\delta < \min \{\k, k+1\}$ and $p < 1+\delta N_p(1,u,a)$, 
	\item  $\k < \min \{ \delta, k+1\}$ and $p <1+ \k N_p(1,u,a)$, 
	
	\item    $\k = \delta \le k  $ and $p<1+ \k N_p(1,s,b)[m, \frac{b}{a+(-a)^{p^k}}]$, 
	\item   $k+1 \le \min \{\k, \delta\}$ and $N_p(d) <1+ (k+1) N_p(1,u,a)$, 
\end{enumerate}
then $K$ is not monogenic.

\end{corollary}
	\begin{remark}
	If one of the  conditions in the previous theorem is satisfied, then the polynomial $F(x)$ is $p$-regular. Then, we can  calculate a $p$-integral basis of $K$. So, we can construct the $p$-index form equation $I_2(x_2, x_3, \ldots, x_n) = \pm 1$ which is a Diophantine equation of degree $\frac{n(n-1)}{2}$ with $n-1$ variables, where the coefficients are  in $\Z_{(p)}$;  the localization of $\Z$ at $p$. However, to make a decision about the monogenity of $K$, one needs to solve index form equation which is not a simple task. Indeed, one must use advanced techniques and methods   in addition to computations using powerful computers and algorithms (see \cite{DS}, and \cite{G19}). Due to all these reasons, we have chosen the prime ideal factorization method as our approach.

\end{remark}

\section{Preliminaries}

To prove  our main results, we need some preliminary results that can be found in details in \cite{BFT1}. \\
For any $\eta \in \Z_K$, we denote by $ ( \Z_K : \Z [\eta])$ the index of $\eta$ in $\Z_K$, where $\Z[\eta]$ is the $\Z$-module generated by $\eta$. It is well known from  \cite[Proposition 2.13]{Na} that   
\begin{eqnarray}\label{indexdiscrininant}
 D (\eta) = ( \Z_K : \Z [\eta])^2 D_K,
\end{eqnarray}
where $D(\eta)$ is the discriminant of the minimal polynomial of $\eta$ and $D_K$ is the discriminant of $K$. In 1878, Dedekind gave the explicit factorization of $p\Z_K$ when $p $ does not divide the index $ ( \Z_K : \Z [\eta])$ (see  \cite{R} and \cite[Theorem 4.33]{Na}). He also  gave  a criterion known as   Dedekind's  criterion to test  the divisibility of  the index $(\Z_K:\Z[\eta])$ by $p$ (see \cite[Theorem 6.14]{Co}, \cite{R}, and  \cite{Na}).
Recall that the index  of the field  $K$ is the following quantity:
\begin{eqnarray*}
i(K) = \gcd \ \{ ( \Z_K : \Z [\eta]) \, |\,\eta \in \Z_K \,and \, K= \Q(\eta)  \}.
\end{eqnarray*}  A rational prime integer $p$ dividing $ i(K)$ is called a prime common index divisor of $K$.  If $K$ is monogenic, then $i(K) = 1$. Thus, a field possessing a  prime common index divisor is not monogenic. The existence of common index divisor was first established by Dedekind. He used the above mentioned criterion and his factorization  theorem to show that the cubic number field $K=\Q(\th)$, where $\th$ is a root of $x^3+x^2-2x+8$ cannot be monogenic, since the rational prime integer   $2$ splits completely in $\Z_K$.  Further, he also  gave in \cite{R} a necessary and sufficient condition  on  a given rational  prime integer  $p$ to be a common index divisor of $K$. This condition depends upon the factorization of the prime  $p$ in $\Z_K$ (see also \cite{He, Hen}). If $p$ does not divide $i(K)$, then there exist $\eta \in \Z_K$ such that $p$ does not divide the index $( \Z_K : \Z [\eta])$. So,  by Dedekind's theorem,  we explicitly  factorize  $p\Z_K$; it is  analogous to the factorization of the minimal polynomial $P_{\eta}(x)$ of $\eta$  modulo $p$. But, if $p$ divides the index $ i(K)$, then the factorization of $p\Z_K$ is more difficult. Hensel  proved in \cite{Hecor} that the prime ideals of $\Z_K$ lying above $p$ are in one-to-one correspondence with irreducible factors of $F(x)$ in $\Q_p(x)$. In 1928,   
O. Ore  developed a method for factoring $F(x)$ in $\Q_p(x)$, and so, factoring $p$ in $\Z_K$ when $F(x)$ is $p$-regular (see \cite{O}). This  method is  based on Newton polygon techniques. So, let us  recall   some fundamental facts on  Newton polygon techniques  applied on prime ideal factorization. {For more details,  refer to \cite{EMN, Nar,  MN92, O}}.
Let $p$  be a rational prime integer and $\n_p$  the discrete valuation of $\Q_p(x)$  defined on $\Z_p[x]$ by $$\n_p(\sum_{i=0}^{m} a_i x^i) = \min \{ \n_p(a_i), \, 0 \le i \le m\}.$$ Let $\phi \in \mathbb{Z}[x]$ be a monic polynomial whose reduction modulo $p$ is irreducible. Any monic irreducible polynomial $F(x) \in \mathbb{Z}[x]$ admits a unique $\phi$-adic development $$ F(x )= a_0(x)+ a_1(x) \phi(x) + \cdots + a_n(x) {\phi(x)}^n$$ with $ deg \ ( a_i (x))  < deg \ ( \phi(x))$. For every $0\le i \le n,$   let $ u_i = \n_p(a_i(x))$. The $\phi$-Newton polygon of $F(x)$ with respect to $p$ is the  lower boundary convex envelope of the set  of points $ \{  ( i , u_i) \, , 0 \le i \le n \, , a_i(x) \neq 0  \}$ in the Euclidean plane, which we denote   by $N_{\phi} (F)$. The polygon  $N_{\phi} (F)$ is the union of different adjacent sides $ S_1, S_2, \ldots , S_g$ with increasing slopes $ \lambda_1, \lambda_2, \ldots,\lambda_g$. We shall write $N_\phi(F) = S_1+S_2+\cdots+S_g$. The polygon determined by the sides of negative slopes of $N_{\phi}(F)$ is called the  $\phi$-principal Newton polygon of $F(x)$  with respect to $p$ and will be denoted by $\npp{F}$. The length of $\npp{F}$ is $ l(\npp{F}) = \nu_{\overline{\ph}}(\overline{F(x)})$;  the highest power of $\phi$ dividing $F(x)$ modulo $p$.\\
Let $\mathbb{F}_{\phi}$ be the finite field   $ \mathbb{Z}[x]\textfractionsolidus(p,\phi (x)) \simeq \mathbb{F}_p[x]\textfractionsolidus (\overline{\ph}) $.
We attach to  any abscissa $ 0 \leq i \leq  l(\npp{F})$ the following residue coefficient $ c_i \in  \mathbb{F}_{\phi}$
:

$$c_{i}=
\left
\{\begin{array}{ll} 0,& \mbox{if }  (i , u_i ) \, \text{ lies  strictly  above }  \ \npp{F},\\
\left(\dfrac{a_i(x)}{p^{u_i}}\right)
\,\,
\md{(p,\phi(x))},&\mbox{if } \  (i , u_i ) \, \text{lies on } \npp{F}. 
\end{array}
\right.$$ Now, let $S$ be one of the sides of $\npp{F}$ and $\lambda = - \frac{h}{e} $  its slope, where $e$ and $h$ are two positive coprime integers. The length of $S$, denoted by $l(S)$ is the length of its  projection to the horizontal axis. The degree of $S$ is $ d = d(S) = \frac{l(S)}{e}$; it is equal to the number of segments into which the integral lattice divides $S$. More precisely, if $ (s , u_s)$ is the initial point of $S$, then the points with integer coordinates  lying in $S$ are exactly $$  (s , u_s) ,\ (s+e , u_s - h) ,  \ldots, \, \mbox{and}\, (s+de , u_s - dh).$$ We attach to $S$ the following residual polynomial:  $$ R_{\l}(F)(y) = c_s + c_{s+e}y+ \cdots + c_{s+(d-1)e}y^{d- 1}+ c_{s+de}y^d \in \mathbb{F}_{\phi}[y].$$  The $\ph$-index of $F(x)$, denoted by $ind_{\ph}(F)$, is  deg$(\ph)$ multiplied by  the number of points with natural integer coordinates that lie below or on the polygon $\npp{F}$, strictly above the horizontal axis  and strictly beyond the vertical axis (see  \cite[Def. 1.3]{EMN} and FIGURE $1$).The polynomial $F(x)$ is said to be $\phi$-regular with respect to $p$ if for each side $S_k$ of $\npp{F}$, the associated residual polynomial $ R_{\l_k}(F)(y)$ is separable in $\mathbb{F}_\phi[y]$. The polynomial $F(x)$ is said to be $p$-regular if it is $\phi_i$-regular for every $ 1 \leq i \leq t$, where $\overline{F(x)}=\prod_{i=1}^t\overline{\ph_i}^{l_i}$is the factorization  of $\overline{F(x)}$ into a product  of powers of distinct monic irreducible polynomials in  $\mathbb{F}_p [x]$. For every $i=1,\dots,t$, let  $N_{\ph_i}^+(F)=S_{i1}+\dots+S_{ir_i}$ and for every {$j=1,\dots, r_i$},  let $R_{\l_{ij}}(F)(y)=\prod_{s=1}^{s_{ij}}\psi_{ijs}^{n_{ijs}}(y)$ be the factorization of $R_{\l_{ij}}(F)(y)$ in $\F_{\ph_i}[y]$. 
By the corresponding statements of the product, the polygon, and  the residual polynomial (see \cite[Theorems 1.13, 1.15 and 1.19]{Nar}),    we have the following  theorem of Ore, which will be often used in the proof of our theorems (see   \cite[Theorem 1.7 and Theorem 1.9]{EMN},  \cite{MN92} and \cite{O}):

\begin{theorem}\label{ore} (Ore's Theorem)
	Under the above notations, we have: 
	\begin{enumerate}
		\item
		$$ \nu_p((\Z_K:\Z[\th]))\ge \sum_{i=1}^t ind_{\ph_i}(F).$$ Moreover, the  equality holds if $F(x)$ is $p$-regular
		\item
		If  $F(x)$ is $p$-regular, then 
		$$p\Z_K=\prod_{i=1}^t\prod_{j=1}^{r_i}
		\prod_{s=1}^{s_{ij}}\pF^{e_{ij}}_{ijs},$$ where $e_{ij}$ is the ramification index
		of the side $S_{ij}$ and $f_{ijs}=\mbox{deg}(\ph_i)\times \mbox{deg}(\psi_{ijs})$ is the residue degree of $\mathfrak{p}_{ijs}$ over $p$.
	\end{enumerate}
\end{theorem}

\begin{example} Let $K = \Q(\th)$, where $\th$ is a root of the monic polynomial $F(x)=x^5+3x^4+24$. Since $F(x)$ is  $3$-Eisenstein polynomial, then it is irreducible over $\Q$. The factorization of  $F(x)$ in  $\F_2[x]$ is $\ol{F(x)}= \ol{(x+1)} \ol{x}^4$. The $x$-Newton polygon of $F(x)$ has a single  side $S$ of degree $1$ joining the points $(0,3)$ and $(4,0)$ with slope $\l = \frac{-3}{4}$ (see FIGURE 1). Thus, the residual polynomial $R_{\l}(F)(y) \in \F_{x}[y] \simeq \F_2[y]$ is irreducible as it is of degree $1$. It follows that the polynomial $F(x)$ is $2$-regular. By Theorem \ref{ore}, we have $$2\Z_K = \pF^4 \qF \, \mbox{ with }\,  f(\pF/2)= f(\qF/2)=1,$$ and $$\n_2((\Z_K:\Z[\th]))= ind_{x}(F)+ ind_{x+1}(F)= 3+0= 3.$$ 
\end{example}

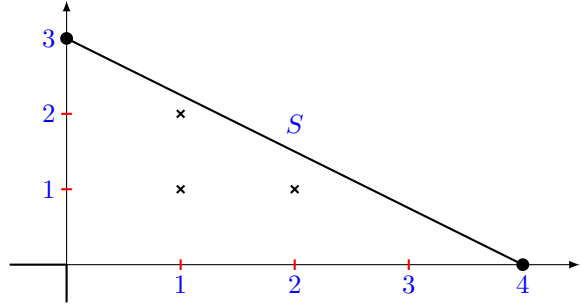
\begin{figure}[htbp]

	\centering
	
	\begin{tikzpicture}[x=1.5cm,y=1cm]
	\draw[latex-latex] (0,3.5) -- (0,0) -- (4.5,0) ;

	\draw[thick] (0,0) -- (-0.5,0);
	\draw[thick] (0,0) -- (0,-0.5);
	
	\draw[thick,red] (1,-2pt) -- (1,2pt);
	\draw[thick,red] (2,-2pt) -- (2,2pt);
	\draw[thick,red] (3,-2pt) -- (3,2pt);
	\draw[thick,red] (4,-2pt) -- (4,2pt);
	\draw[thick,red] (-2pt,1) -- (2pt,1);
	\draw[thick,red] (-2pt,2) -- (2pt,2);
	\draw[thick,red] (-2pt,3) -- (2pt,3);
	\node at (1,0) [below ,blue]{\footnotesize  $1$};
	\node at (2,0) [below ,blue]{\footnotesize $2$};
	\node at (3,0) [below ,blue]{\footnotesize  $3$};
	\node at (4,0) [below ,blue]{\footnotesize  $4$};
	\node at (0,1) [left ,blue]{\footnotesize  $1$};
	\node at (0,2) [left ,blue]{\footnotesize  $2$};
	\node at (0,3) [left ,blue]{\footnotesize  $3$};
	\draw[thick, mark = *] plot coordinates{(0,3)  (4,0)};
	\draw[thick, only marks, mark=x] plot coordinates{(1,1) (1,2)  (2,1)     };
	\node at (2,1.6) [above  ,blue]{\footnotesize $S$};
	\end{tikzpicture}
	\caption{ The $\ph$-principal Newton polygon  $\npp{F}$ with respect to $\n_2$.}
\end{figure}

	Since it is difficult to find the $\ph$-adic development of  the trinomial  $ F(x)$, we will use any adequate  $\ph$-admissible development of $F(x)$. This technique will allow us to comfortably apply Theorem \ref{ore}. In what follows, we recall some useful facts concerning this technique.  Let
\begin{equation}\label{admissdev}
	F(x) = \sum_{j=0}^{n}A_j(x) \ph(x)^j,  \, \,  A_j(x) \in \mathbb{Z}_p[x]
\end{equation}
be a  $\ph$-development of $F(x)$, not necessarily the $\ph$-adic one. Take $ \omega_j = \nu_p(A_j(x))$ for all $ 0 \leq j \leq n$. Let $N$ be the principal Newton polygon of the set of points $\{ ( j, \omega_j ), \, \,  \,  0 \leq j \leq n, \omega_j \neq \infty \}$. To any $ 0 \leq j \leq n$, we attach the following residue coefficient:  $$c^{'}_{j}=
\left
\{\begin{array}{ll} 0,& \mbox{ if } (j,{\it \omega_j}) \mbox{ lies strictly
	above } N
,\\
\left(\dfrac{A_{j}(x)}{p^{\omega_j}}\right)
\,\,
\md{(p,\phi(x))},&\mbox{ if }(j,{\it \omega_j}) \mbox{ lies on }N.
\end{array}
\right.$$
Moreover,  for any side $S$ of $N$ with slope $\l$, we  define the residual polynomial associated to $S$ and denoted by $R_{\l}^{'}(F)(y)$ (similar to the residual polynomial $R_{\l}(F)(y)$ defined from the $\ph$-adic development of $F(x)$). We say that the $ \ph$-development  (\ref{admissdev}) of $F(x)$ is admissible if $ c^{'}_{j} \neq 0 $ for each abscissa $j$ of a vertex of $N$. Recall  that   
$ c^{'}_{j} \neq  0 $ if and only if $ \overline{\ph(x)} $ does not divide $ \overline{\left(\dfrac{A_{j}(x)}{p^{\omega_j}}\right)}.$  For more details, refer to  \cite{Nar}.   The following lemma shows an important relationship between the $\ph$-adic development and any $\ph$-admissible development of a given polynomial $F(x)$.
\begin{lemma} \label{admiss}$($\cite[Lemma $1.12$]{Nar}$)$ \\If a $ \ph$-development of $F(x)$ is admissible, then $ \npp {F} = N$ and $ c^{'}_j = c_j $. In particular, for any segment $S$ of $N$ with slope $\l$, we have $R_{\l}^{'}(F)(y) =R_{\l}(F)(y)$ (up to multiply by a non-zero element of $\F_{\ph}$).  
\end{lemma}

For the determination of certain Newton polygons, we will need to use the following lemmas.
\begin{lemma} \label{binomial}  $($\cite[Lemma 3.4]{BF}$)$\\
	Let $p$ be a rational prime integer and $r$  a positive integer. Then\begin{eqnarray*}
	\nu_p\left(\binom{p^r}{j}\right)  =  r - \nu_p(j) 
	\end{eqnarray*}for any integer $j= 1,\dots,p^r-1 $. 
\end{lemma}

\begin{lemma}\label{lemtech}$($\cite[Lemma 4.1]{BFT1}$)$\\
	Let $p$ be a rational prime integer and $F(x) \in \Z[x] $ be a monic polynomial which is separable modulo $p$. Let $g(x)$ be a monic irreducible factor of $\ol{F(x)}$ in $\F_p[x]$. Then, we can select a monic lifting $\ph(x) \in \Z[x]$ of $g(x)$ (this means that $\ol{\ph(x)} = g(x)$) such that  $$ F(x)= \ph(x) U(x)+pT(x)$$ for some polynomials $U(x)$ and $T(x) \in \Z[x]$ such that $\ol{\ph(x)}$ does not divide $ \overline{U(x)T(x)}$. 
\end{lemma}
 The following lemma gives a sufficient  condition for a rational prime integer $p$ to be a prime common index divisor of a given  field $K$. For the proof, see \cite{R} and \cite[Theorems 4.33 and 4.34 ]{Na}.
\begin{lemma} \label{comindex}
	Let  $p$ be a  rational prime integer and $K$  a number field. For every positive integer $d$, let $P_d$  be the number of distinct prime ideals of $\Z_K$ lying above $p$ with residue degree $d$. If $ P_d > N_p(d)$ for some positive integer $d$, then $p$ is a prime common index divisor of $K$.
\end{lemma}
To apply the last lemma, one needs to know the number $N_p(m)$ of monic irreducible polynomials over $\F_p$ of degree $m$ which  is given by the following proposition. 
\begin{proposition}$($\cite[Proposition 4.35]{Na}$)$
	The number of  monic irreducible polynomials of degree $m$ in $\F_p[x]$ is given by:
	\begin{eqnarray*}
		N_p(m) = \frac{1}{m} \sum_{d \mid m} \mu (d) p^{\frac{m}{d}},
	\end{eqnarray*}where $\mu$ is the M\"{o}bius function. 
\end{proposition}

\section{Proofs of main results}

In this section, we prove our results. Let us start by Theorem \ref{dn1}.

\begin{proof}[Proof of Theorem \ref{dn1}]
 In all cases, we prove that $K$ is not monogenic by showing that $p$ divides $i(K)$. For this reason, in view of Lemma \ref{comindex}, it is sufficient to show that the prime ideal factorization of $p\Z_K$ satisfies the inequality $ P_d > N_p(d)$ for some positive integer $d$.
Under the hypothesis  of  Theorem \ref{dn1}, we have   $$\ol{F(x)} =\ol{ (x^s+b)}^{p^r} \, \mbox{in} \,\, \F_p[x].$$Since $p$ does not divide $ sb$, the polynomial $x^s+\ol{b}$ is separable in $\F_p[x]$. Let $g(x)$ be a monic irreducible factor of  $x^s+\overline{b}$ in $\F_p[x]$ of degree $d$. Using Lemma \ref{lemtech}, we can select a monic lifting $\ph(x) \in \Z[x]$ of $g(x)$ such that  $$ x^s+b= \ph(x) U(x)+pT(x)$$ for some polynomials $U(x)$ and $T(x) \in \Z[x]$ such that $\overline{\ph(x)}$ does not divide $\overline{U(x)}\overline{T(x)}$. Set $M(x)=pT(x)-b$ and write 
\begin{eqnarray*}
F(x)&=& x^n+ax^{m}+b \nonumber \\
&=&(x^s)^{p^r}+ax^m+b \nonumber \\
&=&(\ph(x) U(x)+M(x))^{p^r}+ax^{m}+b \nonumber \\
&=& (\ph(x) U(x))^{p^r}+ \sum_{j=1}^{p^r-1} \binom{p^r}{j} M(x)^{p^r-j}U(x)^j \ph(x)^j + M(x)^{p^r} +ax^{m}+b.\nonumber
\end{eqnarray*}
Applying the binomial theorem, we have

\begin{eqnarray*}
	M(x)^{p^r}&=& (pT(x)-b)^{p^r} \\
	&=&p^{r+1}b^{p^r-1}T(x)+\sum_{j=0}^{p^r-2}(-1)^j\binom{p^r}{j}b^j (pT(x))^{p^r-j}+(-b)^{p^r} \\
	&=& p^{r+1}H(x)+(-b)^{p^r},
\end{eqnarray*}
where $$H(x)= b^{p^r-1}T(x)+\frac{1}{p^{r+1}}\sum_{j=0}^{p^r-2}(-1)^j\binom{p^r}{j}b^j (pT(x))^{p^r-j}. $$It follows that
  \begin{eqnarray*}
F(x)= (\ph(x) U(x))^{p^r}+ \sum_{j=1}^{p^r-1} \binom{p^r}{j} M(x)^{p^r-j}U(x)^j \ph(x)^j +p^{r+1} H(x)+ax^{m}+(-b)^{p^r}+b.
\end{eqnarray*}
Thus 
  \begin{eqnarray}\label{devn1}
F(x) =\sum_{j=0}^{p^r} A_j(x)\ph(x)^j,
\end{eqnarray}
 where 
\[  \begin{cases}
A_0(x) = p^{r+1} H(x)+ax^{m}+(-b)^{p^r}+b,\,\\
A_j(x) =\displaystyle \binom{p^r}{j} M(x)^{p^r-j}U(x)^j \mbox{for every } 1 \leq j \leq p^r.
\end{cases} \]
Using Lemma \ref{binomial}, we get  $$\omega_j = \n_p(A_j(x)) =  \nu_p(\binom{p^r}{j} M(x)^{p^r-j}U(x)^j)= \nu_p(\binom{p^r}{j})= r-\n_p(j)$$ and  $$\overline{\left(\dfrac{A_j(x)}{p^{\omega_j}}\right)}=   \dbinom{p^r}{j}_p  \overline{M(x)^{p^r-j}U(x)^j}$$   for every $1 \leq j \leq p^r$. Moreover,  $\overline{\ph(x)} $ does not divide $ \overline{\left(\dfrac{A_j(x)}{p^{\omega_j}}\right)}$   for every $1 \leq j \leq p^r$, since  $\overline{M(x) }= -\ol{b}\neq \bar{0} $ and  $\overline{\ph(x)} $ does not divide $ \overline{U(x)}$. So,   the 	above $\ph$-development (\ref{devn1}) of $F(x)$ is admissible if and only if $\ol{\ph(x)}$ does not divide $\overline{\left(\dfrac{A_0(x)}{p^{\omega_0}}\right)}$, where $\omega_0 = \n_p(A_o(x))$.
\begin{enumerate}
\item Suppose now that  $\mu < \min\{ \nu , r+1\}$, then $$\omega_0 = \n_p(A_o(x)) = \n_p( p^{r+1} H(x)+ax^{m}+(-b)^{p^r}+b)= \mu$$ and $$\overline{\left(\dfrac{A_0(x)}{p^{\omega_0}}\right)} = a_p x^{m}.$$ Hence, $\overline{\ph(x)} $ does not divide $\overline{\left(\dfrac{A_0(x)}{p^{\omega_0}}\right)}$, because $\ol{\ph(x)} \neq x $. It follows that the $\ph$-development (\ref{devn1}) of $F(x)$ is admissible. For every $k:=1, \ldots,N_p(d,s,b) $, let $\ph_k(x)$ be a monic lifting provided by Lemma \ref{lemtech} of some factor $g_k(x)$ of $\ol{F(x)}$ of degree $d$.  By Lemma \ref{admiss},	$N_{\ph_k}^{+}{(F)} = S_{k1}+ \cdots + S_{k\mu}$ has $\mu$ sides of degree $1$ each joining the points  $(0,\mu), (p^{r-\mu+1}, \mu - 1), \ldots, \, \mbox{and}\, (p^r,0)$ with respective slopes $\l_{k1}= \dfrac{-1}{p^{r-\mu+1}} = \dfrac{-1}{e_1}$,  $\l_{ki}= \dfrac{-1}{(p-1)p^{r-\mu+i-1}} = \dfrac{-1}{e_{i}}$ for every $i = 2, \ldots, \mu$ (see FIGURE 2). 	Thus, $R_{\l_{ki}}(F)(y)$ is irreducible in $\F_{\ph}[y]$ as it is of degree $1$. Then, $F(x)$ is $p$-regular. Applying  Theorem \ref{ore}, we see that $$p\Z_K =  \prod_{k=1}^{N_p(d, s, b)} \prod_{i=1}^{\mu} \pF_{ki}^{e_i} \aF,$$ where $\aF$ is a proper ideal of $\Z_K$, $\pF_{ki}$ is a prime ideal of $\Z_K$ of residue degree $f(\pF_{ki}/p) = \deg(\ph(x)) \times  \deg(R_{\l_{ki}}(F)(y)) = d$.  Then, the $N_p(d,s,b)$ monic irreducible factors of $F(x)$ of degree $d$ modulo $p$ provides  $\mu N_p(d,s,b)$ prime ideals of residue degree $d$ each. By using Lemma \ref{comindex},  if $N_p(d)<\mu N_p(d,s,b)$, then $p $ divides $ i(K)$, and so $K$ is not monogenic.
 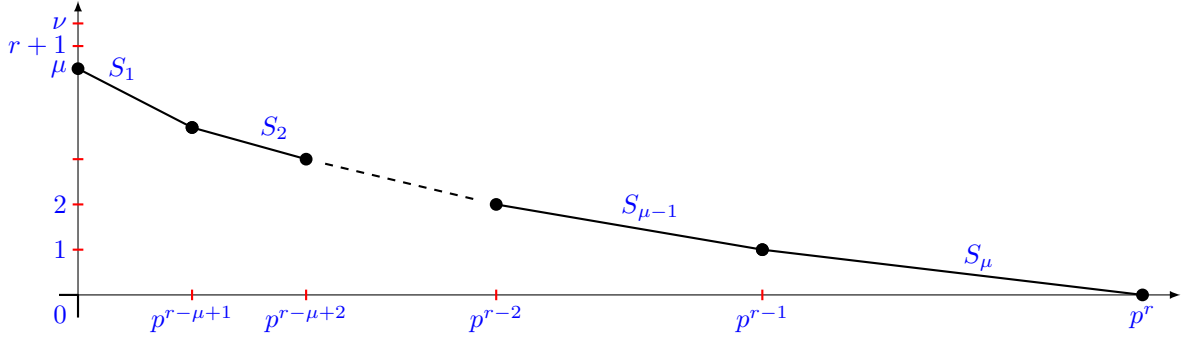
\begin{figure}[htbp] 
	\centering
	\begin{tikzpicture}[x=0.5cm,y=0.6cm]
	\draw[latex-latex] (0,6.5) -- (0,0) -- (29,0) ;
	\draw[thick] (0,0) -- (-0.5,0);
	\draw[thick] (0,0) -- (0,-0.5); 
	\draw[thick,red] (3,-2pt) -- (3,2pt);
	\draw[thick,red] (6,-2pt) -- (6,2pt);
	\draw[thick,red] (11,-2pt) -- (11,2pt);
	\draw[thick,red] (18,-2pt) -- (18,2pt);
	\draw[thick,red] (28,-2pt) -- (28,2pt);
	\draw[thick,red] (-2pt,1) -- (2pt,1);
	\draw[thick,red] (-2pt,2) -- (2pt,2);
	\draw[thick,red] (-2pt,3) -- (2pt,3);
	\draw[thick,red] (-2pt,5) -- (2pt,5);
	\draw[thick,red] (-2pt,5.5) -- (2pt,5.5);
		\draw[thick,red] (-2pt,6) -- (2pt,6);
	\node at (0,0) [below left,blue]{\footnotesize  $0$};
	\node at (3,0) [below ,blue]{\footnotesize $p^{r-\mu+1}$};
	\node at (6,0) [below ,blue]{\footnotesize  $p^{r-\mu+2}$};
	\node at (11,0) [below ,blue]{\footnotesize  $p^{r-2}$};
	\node at (18,0) [below ,blue]{\footnotesize  $p^{r-1}$};
	\node at (28,0) [below ,blue]{\footnotesize  $p^{r}$};
	\node at (0,1) [left ,blue]{\footnotesize  $1$};
	\node at (0,2) [left ,blue]{\footnotesize  $2$};
	\node at (0,5) [left ,blue]{\footnotesize  $\mu$};
	\node at (0,6) [left ,blue]{\footnotesize  $\n$};
		\node at (0,5.5) [left ,blue]{\footnotesize  $r+1$};
	\draw[thick,mark=*] plot coordinates{(0,5) (3,3.7)};
	\draw[thick,mark=*] plot coordinates{(3,3.7) (6,3)};
	\draw[thick,mark=*] plot coordinates{(11,2) (18,1)};
	\draw[thick,mark=*] plot coordinates{(28,0) (18,1)};
	\draw[thick, dashed] plot coordinates{(6.5,2.9) (10.5,2.1) };
	\node at (0.5,4.5) [above right  ,blue]{\footnotesize  $S_{1}$};
	\node at (4.5,3.2) [above right  ,blue]{\footnotesize  $S_{2}$};
	\node at (14,1.4) [above right  ,blue]{\footnotesize  $S_{\mu-1}$};
	\node at (23,0.3) [above right  ,blue]{\footnotesize  $S_{\mu}$};
	\end{tikzpicture}
	\caption{    \large  $\npp{F}$ with respect to $p$ when  $\mu < \min\{ \nu, r+1\}$.\hspace{5cm}}
\end{figure}

\item If  $\nu < \min\{ \mu, r+1\}$, then $\omega_0= \n $ and $$\overline{\left(\dfrac{A_0(x)}{p^{\omega_0}}\right)} = (b+(-b)^{p^r})_p.$$ So,  $\overline{\ph(x)} $ does not divide $ \overline{\left(\dfrac{A_0(x)}{p^{\omega_0}}\right)}$. Thus, the $\ph$-development (\ref{devn1}) of $F(x)$ is admissible. Now, we proceed similarly to the proof of the first point, we show that $p\mid i(K)$.
\item  If  $\mu = \n \le r  $, then  $\omega_0 = \n = \mu$ and $$\overline{\left(\dfrac{A_0(x)}{p^{\omega_0}}\right)} =a_p x^m +  (b+(-b)^{p^r})_p.$$ In this case, we consider the monic irreducible factors of $x^s+\ol{b}$ of degree $d$ in $\F_p[x]$ which do not divide $ \ol{a_p }x^m +  \ol{(b+(-b)^{p^r})_p}$ which guarantee that $\ol{\ph(x)}$ does not divide $\overline{\left(\dfrac{A_0(x)}{p^{\omega_0}}\right)}$. The number of these factors is $ N_p(d,s,b)[m, \frac{b+(-b)^{p^r}}{a}]$.
\item Now, we deal with the case when  $ r+1  \le \min\{ \nu, \mu\}$. Let $$A_0(x)= \sum_{j = 0}^{l} a_j(x) \ph(x)^j$$ be the $\ph$-adic development of the polynomial $A_0(x)$  ($ \deg(a_j(x)) <  \deg(\ph(x))$) with $l = \lfloor \dfrac{\deg(A_0(x))}{d} \rfloor $ which can exceed $ p^r $ according to the degree of the polynomial $T(x)$ (see the above expressions of $H(x)$ and $A_0(x)$).    Substituting in (\ref{devn1}), we see that 
\begin{eqnarray}\label{devn12}
F(x) = \sum_{j=0}^{l}  B_j(x) \ph(x)^j,
\end{eqnarray}
where
\[  \begin{cases}
B_0(x)=a_0(x),\,\\
B_j(x)= \displaystyle \binom{p^r}{j} M(x)^{p^r-j}U(x)^j + a_j(x)= A_j(x)+a_j(x)\, \mbox{for every } \,j=1, \ldots, p^r,\,\\
B_j(x)= a_j(x) \,\, \mbox{if} \, j \ge p^r+1.
\end{cases} \]
 Since $l(\npp{F})= \n_{\ol{\ph(x)}}(\ol{F(x)})=p^r$, we are  only interested in the first $p^r$ coefficients in the above $\ph$-development  (\ref{devn12}) of $F(x)$.	Since  $ r+1  \le \min\{ \nu, \mu\}$,  $\n_p(A_0(x)) \ge r+1$, and so $\n_p(a_j(x)) \ge r+1$ for every $j=0, \ldots, l$. On the other hand, as $\n_p(A_j)  \le r$,  $$\omega_j^{'}=\n_p(B_j(x)) = \n_p(A_j(x)) = \omega_j = \n_p(\binom{p^r}{j}) = r-\n_p(j),$$ $$\overline{\left(\dfrac{B_j(x)}{p^{\omega_j^{'}}}\right)} = \overline{\left(\dfrac{A_j(x)}{p^{\omega_j}}\right)},$$ and so $\overline{\ph(x)} $ does not divide $ \overline{\left(\dfrac{B_j(x)}{p^{\omega_j^{'}}}\right)}$ for every $j=1, \ldots, p^r$. We have also, $\overline{\ph(x)} $ does not divide $ \overline{\left(\dfrac{B_0(x)}{p^{\omega_0^{'}}}\right)}$, because $\deg (a_0(x)) < \deg(\ph(x))$, where   $\omega_0^{'} = \n_p(B_0(x))=\n_p(a_0(x))\ge r+1$. It follows that the $\ph$-development (\ref{devn12})   of $F(x)$ is admissible. By Lemma \ref{admiss}, $\npp{F}= S_o+ \cdots+S_r$ has $r+1$ sides of degree $1$ each joining the points $(0, \omega_0^{'}), (1,r), (p,r-1), \ldots, \, \mbox{and}\, (p^r,0)$ with respective slopes $\l_0 = \omega_0^{'}-r =  \dfrac{\omega_0^{'}-r}{e_1} \le -1$, $\l_i = \dfrac{-1}{p^i-p^{i-1}}= \dfrac{-1}{e_i}$ for every $i=1, \ldots, r$.  By Theorem \ref{ore}, we get $$p \Z_K = \prod_{k=1}^{N_p(d, s, b)} \prod_{i=0}^{r} \pF_{ki}^{e_i} \bF,$$ where $\bF$ is a proper ideal of $\Z_K$, and $\pF_{ki}$ is a prime ideal of $\Z_K$ with residue degree $f(\pF_i/p)= d $.   By Lemma \ref{comindex}, if $N_p(d) <(r+1) N_p(s,d,b)$, then $p $ divides $ i(K)$. Consequently,  $K$ is not monogenic.  
\end{enumerate}

\end{proof}

\begin{remark} In Theorem \ref{dn1}(4), if  $ r+1  < \min\{ \nu, \mu\}$, then $\npp{F}$ is exactly the Newton polygon joining the points  $(0, r+1), (1,r), (p,r-1), \ldots, \, \mbox{and}\, (p^r,0)$.  Let us explain this fact.
   By using lemma \ref{binomial}, we have 
 $$\n_p(\binom{p^r}{j}b^j (pT(x))^{p^r-j}) = r-\n_p(j)+p^r-j  \ge r+2$$ for every $1 \le j \le p^r-2$ (it suffice to distinct the two cases when $p$ divides or not $j$). It follows that $\ol{H(x)} = \ol{b^{p^r-1}T(x)} \neq \ol{0}$. Therefore, $\omega_0= \n_p(A_0(x))= r+1$. So,  $$\overline{\left(\dfrac{A_0(x)}{p^{\omega_0}}\right)} = \ol{H(x)} = \ol{b^{p^r-1}T(x)}.$$
Consequently, the $\ph$-adic development  (\ref{devn1}) is admissible.   In this case, we don't need to consider the $\ph$-adic development of $A_0(x)$ to determine $\npp{F}$.
 
\end{remark}
\begin{proof}[Proof of Theorem \ref{dn2}]
Under the assumptions of Theorem \ref{dn2}, we see that $$F(x) \equiv  x^m(x^{n-m}+a) \equiv x^m (x^u+a)^{p^k} \md p.$$  Since $p$  does not divide $ au$,  the polynomial $x^u+\ol{a}$ is separable in $\F_p[x]$. It follows that if $x^u+\ol{a}= \prod_{i=1}^{t} \ol{\ph_i(x)}$ in  $\F_p[x]$, then $$\ol{F(x)} = x^m (\prod_{i=1}^{t} \ol{\ph_i(x)})^{p^k} \, \mbox{in} \, \F_p[x].$$ Fix $\ph(x) = \ph_i(x)$ for some $1 \le i \le t$ and let  $d=\deg(\ph(x))$.  By Lemma \ref{lemtech}, let $V(x), R(x)\in \Z[x]$ such  that $$x^u + a = \ph(x)V(x)+pR(x)$$ with $\ol{\ph(x)}$ does not divide $\ol{V(x)R(x)}$.
By setting $N(x)= pR(x)-a $ and $$L(x)=  a^{p^k-1}R(x)+\frac{1}{p^{k+1}}\sum_{j=0}^{p^k-2}(-1)^j\binom{p^k}{j}a^j (pR(x))^{p^k-j},$$ we obtain  that 

 \begin{eqnarray*}
F(x)= (\ph(x) V(x))^{p^k}x^m+ \sum_{j=1}^{p^k-1} \binom{p^k}{j}x^m N(x)^{p^k-j}V(x)^j \ph(x)^j +p^{k+1} x^m L(x)+((-a^{p^k})+a)x^{m}+b.
\end{eqnarray*}
It  follows that
\begin{eqnarray}\label{devn2}
F(x) = \sum_{j=0}^{p^k} A_j(x) \ph(x)^j,
\end{eqnarray}
 where 

\[  \begin{cases}
A_0(x) = p^{k+1}x^m L(x)+((-a)^{p^k}+a)x^{m}+b,\,\\
A_j(x) =\displaystyle \binom{p^k}{j}x^m N(x)^{p^K-j}V(x)^j \mbox{for every } 1 \leq j \leq p^k.
\end{cases} \]

\end{proof}
To complete  the proof, we proceed analogously to the proof  of Theorem \ref{dn1}, we show that $p $ divides $ i(K)$. Using 	Lemma \ref{binomial} and the fact that $\ol{\ph(x)}$ does not divide $\ol{R(x)V(x)}$, we see that the above $\ph$-development of $F(x)$ is admissible. By Lemmas \ref{admiss} and \ref{binomial}, we determine $\npp{F}$ and  we get that $F(x)$ is $p$-regular. Applying Theorem \ref{ore} to factorize $p\Z_K$ in the ring  $\Z_K$. This factorization combined with the conditions of Theorem \ref{dn2} satisfies the sufficient condition of Lemma \ref{comindex}.   Consequently, $p$ divides $i(K)$.
\begin{proof}[Proof of Corollaries \ref{cordn1} and \ref{cordn2}] 
For every $k\ge 1$, we have the following factorization:
$$\ol{x^{2^k}-1}=\ol{(x-1)(x-2)U(x)} \, \mbox{in}\, \F_3[x],$$
where $\ol{\ph_i(x)}$ does not divide $\ol{U(x)}$ for $i=1,2$. It follows that $N_3(1,2^k,-1)=2$ for every $k\ge1$. Also, the polynomial $\ol{x^2+1}$ is irreducible in $\F_3[x]$, then $N_3(2,2,1)=1$. On the other hand $$\ol{x^4+1}= \ol{(x^2-x-1)(x^2+x-1)} \, \mbox{in}\, \F_3[x],$$ 
then $N_3(2,2^2,1)=2$. Using these factorizations and a direct application of Theorems \ref{dn1} and \ref{dn2}, we conclude the two corollaries.    
\end{proof} 
\begin{proof}[Proof Theorem \ref{d6}]
	We recall that Corollary \ref{cordn1}(5) implies that if $ a \equiv 0 \md 9\, \mbox{and} \, b \equiv  -1  \md {9}$, then $K$ is not monogenic. Moreover, using the proof of Theorem \ref{dn1}, we see that $$3\Z_K=\pF_1 \pF_2 \pF_3^2 \pF_4^2$$
		with residue degrees $f(\pF_k/3)=1$ for $k=1,2,3,4$. By Using Engstom's table concerning the index of number fields of degrees less than $7$ (see \cite[Page 234]{Engstrom}), we see that $\n_3(i(K))= 1$.
	 Now, we deal with the case when $a\equiv 0\md 8$ and $b\equiv -1 \md 8$. In this case, we have   $\ol{F(x)}= (\ol{\ph_1(x)\ph_2(x)})^2$ in $\F_2[x]$, where  $\ph_1(x)=x-1$ and $\ph_2=x^2+x+1$. Write
	\begin{eqnarray}\label{dev6}
		F(x)&=& x^6+ax^{m}+b \nonumber \\
		&=&(x^3-1+1)^2+ax^m+b \nonumber \\
		&=&(\ph_1(x)\ph_2(x)+1)^2+ax^m+b \nonumber \\
		&=& (\ph_1(x)\ph_2(x))^2+ 2\ph_1(x)\ph_2(x)+ax^m+1+b.
	\end{eqnarray}
Let $\mu=\n_2(a)$, $\n=\n_2(1+b)$, and $a_0(x)=ax^m+1+b$. Since $a\equiv 0\md 8$ and $b\equiv -1 \md 8$,  $\omega_0=\n_2(a_0(x)) = \min\{\mu, \n\} \ge 3$. If $\mu > \n$,  	then  $\omega_0=\n$ and  $\ol{\ph_2(x)}$ does not divide $\ol{(\dfrac{a_0(x)}{2^{\n}})} = \ol{1}$. It follows that the above $\ph_i$-development (\ref{dev6}) of $F(x)$ is admissible for $i=1,2$. By Lemma \ref{admiss}, for  $i=1,2$, $N^{+}_{\ph_i}(F)=S_{i1}+S_{i2}$ has two sides of degree $1$ each joining the points $(0, \n), (1,1),\, \mbox{and}(2,0)$ with respective slopes $\l_{i1} \le -2$ and $\l_{i2}=-1$ (see FIGURE 3). Thus, $R_{\l_{ik}}(F)(y)$ is irreducible over $\F_{\ph_i}$. It follows that the polynomial $F(x)$ is $2$-regular. Applying Theorem \ref{ore}, we obtain that $$2\Z_K = \pF_1 \pF_2 \pF_3 \pF_4$$ with residue degrees $$f(\pF_1/2)=f(\pF_2/2)=1 \,\, \mbox{and}\, f(\pF_3/2)=f(\pF_4/2)=2$$
 Similarly, if $\n > \mu$, we see that  $$2\Z_K = \pF_1 \pF_2 \pF_3 \pF_4$$ with residue degrees $$f(\pF_1/2)=f(\pF_2/2)=1 \,\, \mbox{and}\, f(\pF_3/2)=f(\pF_4/2)=2.$$ Now, we deal with the case when $\mu = \n $.	 Write $$ax^m+1+b = 2^{\mu}(a_2x^m+(1+b)_2)=2^{\mu}(\ph_2U_2(x))+R_2(x)),$$ where $U_2(x)$ and $R_2(x)$ are respectively  the quotient and the remainder upon the Euclidean division of $a_2x^m+(1+b)_2$ by $\ph_2(x)$. 
Substituting in (\ref{dev6}), we get that 
 \begin{eqnarray}\label{dev62}
 F(x)= (\ph_1(x)\ph_2(x))^2+ (2\ph_1(x)+2^{\mu}U_2(x))\ph_2(x)+2^{\mu}R_2(x).
 \end{eqnarray}
	Since $\ol{\ph_2}(x)$ does not divide $\ol{\ph_1(x)R_2(x)}$, the above $\ph_2$-development (\ref{dev62}) of $F(x)$ is admissible. By Lemma \ref{admiss}, $N^{+}_{\ph_2}(F)=S_1+S_2$ has two sides of degree $1$ each joining the points $(0, \mu), (1,1),\, \mbox{and}(2,0)$ with respective slopes $\l_1 \le -2$ and $\l_2=-1$. Applying Theorem \ref{ore}, the irreducible factor $\ph_2(x)$ of $F(x)$ provides two prime ideals of $\Z_K$ lying above $2$ of residue degree $2$ each. Similarly, we see that the irreducible factor $\ph_1$ provides two prime ideals of $\Z_K$ of  residue degree  $1$ each lying above $2$. We conclude that  if $a\equiv 0\md 8$ and $b \equiv -1 \md 8$, then   $$2\Z_K = \pF_1 \pF_2 \pF_3 \pF_4$$ with residue degrees $$f(\pF_1/2)=f(\pF_2/2)=1 \,\, \mbox{and}\, f(\pF_3/2)=f(\pF_4/2)=2,$$ By using the above mentioned Engstrom's table, $\n_2(i(K))=2$. Consequently, $K$ is not monogenic. This complete the proof of the theorem.
\end{proof}

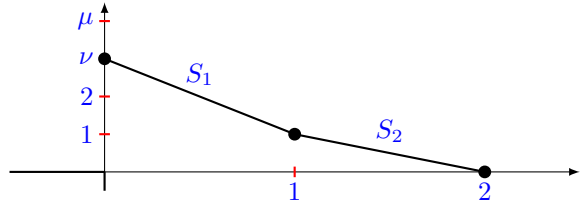
\begin{figure}[htbp]

	\centering
	
	\begin{tikzpicture}[x=2
	.5cm,y=0.5cm]
	\draw[latex-latex] (0,4.5) -- (0,0) -- (2.5,0) ;

	\draw[thick] (0,0) -- (-0.5,0);
	\draw[thick] (0,0) -- (0,-0.5);
	
	\draw[thick,red] (1,-2pt) -- (1,2pt);
	\draw[thick,red] (2,-2pt) -- (2,2pt);
	\draw[thick,red] (-2pt,1) -- (2pt,1);
	\draw[thick,red] (-2pt,2) -- (2pt,2);
	\draw[thick,red] (-2pt,3) -- (2pt,3);
	\draw[thick,red] (-2pt,4) -- (2pt,4);	
	\node at (1,0) [below ,blue]{\footnotesize  $1$};
	\node at (2,0) [below ,blue]{\footnotesize $2$};
	\node at (0,1) [left ,blue]{\footnotesize  $1$};
	\node at (0,2) [left ,blue]{\footnotesize  $2$};
	\node at (0,3) [left ,blue]{\footnotesize  $\n$};
	\node at (0,4) [left ,blue]{\footnotesize  $\mu$};
	\draw[thick, mark = *] plot coordinates{(0,3) (1,1) (2,0)};
	\node at (0.5,2) [above  ,blue]{\footnotesize $S_{1}$};
	\node at (1.5,0.5) [above   ,blue]{\footnotesize $S_{2}$};
	\end{tikzpicture}
	\caption{ The $\ph_i$-principal Newton polygon  $N^{+}_{\ph_i}(F)$ for $i=1,2$ with respect to $\n_2$ when  $\n <\mu$.}
\end{figure}

which is equal to $\frac{1}{d} \sum_{k \mid d} \mu (k) p^{\frac{d}{k}}$, where $\mu$ is the M\"{o}ubius function

\section{Examples}
Let $F(x) \in \Z[x]$ be a	 monic irreducible polynomial  and $K$ a number field generated by a complex root of $F(x)$.

\begin{enumerate}
\item  If $F(x)=x^{18}+342 x^m+26$, then by Corollary \ref{cordn1}(6), $3$ divides $i(K)$, and so, $K$ cannot be monogenic. 
\item By Theorem \ref{d6}, the sextic pure field $K=\Q(\sqrt[6]{63})$ cannot be monogenic.
\item If $F(x)=x^{12}-19x^6+171$, then by Corollary \ref{cordn2}(6),  $3$ divides $i(K)$. Consequently, $K$ is not monogenic.
\item  If $F(x)= x^6+3249x^m+152$, then by Theorem \ref{d6}, $3$ divides $i(K)$. Hence, $K$ is not monogenic.
\end{enumerate}

\end{document}